\documentclass[12pt,reqno]{article}
\usepackage{amsgen, amsmath, amsfonts, amsthm, amssymb, enumerate,amscd,graphics,dsfont,comment,tikz-cd}
\usetikzlibrary{intersections}

\newtheorem{defn}{Definition}[section]
\newtheorem{prop}[defn]{Proposition}
\newtheorem{lem}[defn]{Lemma}
\newtheorem{thm}[defn]{Theorem}
\newtheorem{cor}[defn]{Corollary}
\newtheorem{rem}[defn]{Remark}

\newcommand {\ZZ}{{\mathds Z}}

\newcommand {\C}{{\mathds C}}

\newcommand {\A}{{\mathcal A}}

\newcommand {\Q}{{\mathds Q}}

\newcommand {\OO}{{\mathcal O}}

\newcommand {\M}{{\mathcal M}}

\newcommand {\CP}{{\mathds P}}

\newcommand {\AB}{{\mathds A}}

\newcommand{\hG}{\hat{G}}

\def\Ind{\operatorname{Ind}}

\def\Pic{\operatorname{Pic}}
\def\Ker{\operatorname{Ker}}

\def\div{\operatorname{div}}

\def\mod{\operatorname{mod}}
\def\dim{\operatorname{dim}}
\def\Spec{\operatorname{Spec}}
\def\Im{\operatorname{Im}}
\def\ord{\operatorname{ord}}
\def\reg{\operatorname{reg}}
\def\Hom{\operatorname{Hom}}
\def\Ext{\operatorname{Ext}}
\title{Torsion cycles on Fermat varieties}
\author{Ramesh Sreekantan\\Indian Statistical Institute Bangalore}
\begin{document}
	\baselineskip=17pt
	\maketitle

	\begin{abstract}
		
		A theorem of Manin and Drinfeld \cite{madr},\cite{drin} states that any divisor of degree $0$ on the cusps of a modular curve is torsion in the Jacobian. An elegant proof of this result was provided by Elkik \cite{elki} using mixed Hodge theory. 
		
		Rohrlich \cite{rohr} proved a generalization of this to Fermat curves. In this note we reprove his results along the lines of the work of Elkik. We then use the same ideas to generalize it to certain higher codimensional null homologous cycles and higher Chow cycles on Fermat varieties. 
		
		\vspace{\baselineskip}
		{\bf 2020 Mathematics Subject Classification (MSC 2020):} 14C25, 19E15, 14C30.

	\end{abstract}

	%\tableofcontents
	
	\section{Introduction.}
	
	A well-known theorem of Manin and Drinfeld  \cite{madr,drin} states that any divisor of degree $0$ supported on the cusps of a modular curve is torsion in the Jacobian. This was generalized to Fermat curves by   Rohrlich \cite{rohr}. He proved the following theorem: 
	
	\begin{thm}[Rohrlich] Let $F_d=F_d^1$  be the Fermat curve of degree $d$ defined over a field $K \subset \C$. 
		$$	X_0^d+X_1^d+X_2^d=0.$$
		Let $S^1_d$ be the set of $3d$ `Fermat points' - which are points determined by the hyperplanes $X_i=0$.  Then any divisor of degree $0$ supported on $S^1_d$ is torsion in the Jacobian. That is, if $D=\sum_{P \in S^1_d} a_P P$ such that $\sum_{P \in S^1_d} a_P=0$ and $D$ is defined over $K$ then there exists a function $f \in K(F_d)$ such that for some $k \in {\mathds N}$,
		$$\div(f)=kD$$
	\end{thm}
	
	There is an elegant `pure thought' proof of the original Manin-Drinfeld theorem due to Deligne and Elkik \cite{elki}. The idea is that the Abel-Jacobi map can be described in terms of extensions of mixed Hodge structures. Further, if $X$ is a modular curve, $S$ the set of cusps and $Y=X- S$ the open curve, then the rational Hodge structure on $H^1(Y)$ is a direct sum of the Hodge structure of pure weight $1$ on $H^1(X)$ and a Hodge structure of pure weight $2$. Once one knows this a short argument (Proposition \ref{splitting}) shows that any divisor of degree $0$ supported on $S$ is torsion. 
	
	In this paper  we show how the same idea can be applied to get a new proof of Rohrlich's theorem - though not quite as explicit as his is. However, one advantage is that this can be generalized to higher codimensions with only a little more work, though we get a weaker statement - only torsion in the intermediate Jacobian.

	Let $F_d^n$ be a Fermat hypersurface and $I \subset \{0,1,\dots,n+1\}$. Let 
	$$C_I= F_d^n \bigcap_{i_k \in I} \{X_{i_k}=0\}$$ 
	A {\bf coordinate hyperplane subvariety} of a Fermat hypersurface is a subvariety obtained as a component of a $C_I$ for some $I$. Note that if $\dim(C_I)>0$ then $C_I$ is irreducible, so it is only in the case of points that there are components.

	 We have the following theorem (Theorem \ref{nullhomologoustheorem}, Corollary \ref{nullhomologouscorollary}).

	  \begin{thm} Let $F^n_d$ be the Fermat hypersurface of degree $d$ and dimension $n$. Let $S_d^p=|\cup_{|I|=p} C_I|$ be the support of the union of all the coordinate hyperplane subvarieties of codimension $p$.  Then the exact sequence 
		$$0 \rightarrow H^{2p-1}(F^n_d,\Q) \rightarrow H^{2p-1}(F^n_d-S_d^p,\Q) \rightarrow H^{2p}_{S_d^p}(F_d^n,\Q)^0 \rightarrow 0$$
		is a split exact sequence of Hodge structures. Here 
		$$H^{2p}_{S_d^p}(F_d^n,\Q)^0=\Ker\{H^{2p}_{S_d^p}(F^n_d,\Q) \rightarrow H^{2p}(F^n_d,\Q)\}$$
		is the kernel of the cycle class map and is generated by linear combinations of the coordinate hyperplane subvarieties  which are null-homologous. 
		
		As a result, any null-homologous cycle supported on $S_d^p$ is torsion in the $p^{th}$-intermediate Jacobian. 
	
	\end{thm}
In fact for $n>1$ and $p<n$ it is trivial that a null-homologous  cycle supported on $S_d^p$ is rationally equivalent to $0$ in the Chow group itself  but the $n=1$ case is Rohrlich's theorem. 
	
We then show  that similar methods can be used to show that certain cycles in the higher Chow groups of Fermat varieties are decomposable in a certain strong sense. 

A  cycle in $CH^p(F_d^n,q)_{\Q}$  is said to be {\bf ambient decomposable} (Section \ref{ambientdecomp}) if it lies in the subspace 
$$CH^p_{amb}(F_d^n,q)_{\Q}=\Im \left( \bigoplus_{0 \leq r  \leq n}   CH^{p-r}(\Spec(K),q)_{\Q} \stackrel{\oplus H^r   \pi^*(-)}{\longrightarrow} CH^p(F_d^n,q)_{\Q}\right)$$
where $H$ is the hyperplane section and $\pi$ is the structure map. In other words these are the kind of cycles determined by the embedding and the base field - and not the variety itself. We define a cycle $y$ to be {\bf regulator ambient decomposable} if there is an ambient cycle $y_{amb}$ such that 
$$\reg(y)=\reg(y_{amb}).$$
We then have the following theorem 
\begin{thm} Let $F_d^n$ be the Fermat hypersurface and $1 \leq q \leq p$. Then 
	
	\begin{itemize}
	\item {\bf Case I} $p-q>0$: Let $y$ be a cycle in $CH^p(F_d^n,q)_{\Q}$ such that each irreducible component of $y$ projects on to a proper coordinate hyperplane subvariety $C_I \subset F_d^n$ with $|I| \geq p-q$. Then $y$  is ambient decomposable.
	
	\item {\bf Case II} $p-q=0$, $n=p-1>1$: Let $C=\bigcup_{i=0}^{n+1} C_i$ and $U=F_d^n-C$. Let $y$ be a cycle in $CH^p(F_d^{p-1},p)_{\Q}$ such that 
	$$y|_U=\sum_s c_s\{f_{s_1},\dots,f_{s_p}\}$$
	with $c_s \in \Q$ and $f_{s_k} \in \OO(U)^*$ 
	with $\div(f_{s_k})$ supported on the $C$. Then $y$ is ambient decomposable.

\end{itemize}
\end{thm}

We further observe that if $n \neq 2p-q-1$ and $0 \leq p \le n$, then {\em all} cycles in $CH^p(F_d^n,q)_{\Q}$ are regulator ambient decomposable. 
   
The upshot is that it is only when $n=2p-q-1$ that the group  $CH^p(F_d^n,q)$ can have `interesting' cycles - however, cycles coming from coordinate hyperplane subvarieties do not produce them. In the last section we discuss other methods of possibly constructing them. 
	
	{\bf Acknowledgments:} I would like to thank the Korea Advanced Institute of Science and Technology (KAIST, Daejeon) for their hospitality when this work was started. I would like to acknowledge the use of AI - specifically Gemini and ChatGPT  - for some of the arguments, literature and typographical and stylistic comments. I would like to thank Kannappan Sampath and Jinhyun Park for their comments. Finally, I would like to thank the referee for their comments. 
		
		\section{Hodge structures.}
		
		\subsection{Mixed Hodge structures.}
		
		According to Deligne \cite{deli}, a mixed Hodge structure $H$ is a finitely generated Abelian group   $H_{\ZZ}$ with two filtrations. 
		
		\begin{itemize}
			\item  An increasing  {\em weight filtration} $W_{\bullet}$ on $H_{\Q}=H_{\ZZ} \otimes \Q$. 
			
			\item A decreasing {\em Hodge filtration} $F^{\bullet}$ on $H_{\C}=H_{\Q}\otimes \C$.
		\end{itemize}
	such that the graded pieces $Gr^{W_{\bullet}}_l = W_l/W_{l-1}$ 
    have a pure Hodge structure of weight $l$. The category of  mixed Hodge structures will be denoted by $MHS$. 
    
    A {\em rational} mixed Hodge structure is a finite dimensional $\Q$-vector space satisfying the above two conditions. We denote the category of rational mixed Hodge structures by $\Q-MHS$.

	\subsection{Extensions of Mixed Hodge structures.}
	
	An extension of $B$ by $A$, where $A$ and $B$ are MHSs, is a short exact sequence of mixed Hodge structures  
	$$e(B,A):0 \rightarrow A \rightarrow E \rightarrow B \rightarrow 0$$
	$e(B,A)$ determines a class in the set $\Ext_{MHS}(B,A)=\Ext^1_{MHS}(B,A)$. The set $\Ext_{MHS}^1(B,A)$ is a group under the Baer sum. A class is trivial if and only if the exact sequence splits.

	An extension is called a {\em separated} extension of MHS if the  highest weight of $A$ is less than the lowest weight of $B$. Carlson \cite{carl} showed that in this case the group can be identified with a generalized torus 
	
	\begin{thm}[Carlson] If $A$ and $B$ are  mixed Hodge structures with $A_{\ZZ}$ and $B_{\ZZ}$ torsion free such that the highest weight of $A$ is less than the lowest weight of $B$ - so any $e(B,A)$ in $\Ext(B,A)$ is separated -  then 
		$$\Ext_{MHS}(B,A)=\Hom_{\C}(B,A)/(F^0 \Hom_{\C}(B,A)+\Hom_{\ZZ}(B,A))$$
		
		\end{thm}	 
		
		\subsection{Extensions and the Abel-Jacobi map.}

	In what follows let  $X$ be a smooth projective variety over a field $K \subset \C$. We do not need to assume that $K$ is a number field but in the application to Fermat curves that is certainly the case. 
	
    Let  $D=\sum a_i D_i$ be a codimensional-$p$ cycle on $X$. Let $|D|=\cup D_i$   denote the support of the cycle. Let $Y=X-|D|$ and $i:Y \hookrightarrow X$. Then one has a long exact sequence of cohomology with support 
	$$\dots \rightarrow  H^{2p-1}(X,\ZZ) \stackrel{i^*}{\rightarrow} H^{2p-1}(Y,\ZZ)\rightarrow H^{2p}_{|D|}(X,\ZZ)\rightarrow  H^{2p}(X,\ZZ)\rightarrow \dots$$
	The term $H^{2p}_{|D|}(X,\ZZ)$ is pure of weight $2p$. Hence the cohomology group $H^{2p-1}(Y,\Q)$ has the weight filtration 
	\begin{itemize}
		\item $W_{2p-2}=\{0\} \subset W_{2p-1}=i^*(H^{2p-1}(X,\Q))\subset W_{2p}= H^{2p-1}(Y,\Q)$
		\item $Gr_{2p-1}^{W_{\bullet}}=W_{2p-1}=i^*(H^{2p-1}(X,\Q))$
		\item $Gr_{2p}^{W_{\bullet}}=H^{2p}_{|D|}(X,\Q)^0$, where 
		$$H^{2p}_{|D|}(X,\Q)^0=\Ker\{H^{2p}_{|D|}(X,\Q) \rightarrow H^{2p}(X,\Q)\}$$
		
	\end{itemize}
	
	If $Z$ is a null-homologous cycle supported on $|D|$ then it determines a morphism   $\ZZ(-p) \rightarrow  H^{2p}_{|D|}(X,\ZZ)$. Pulling back   induces a short exact sequence 
	$$0 \rightarrow H^{2p-1}(X,\ZZ) \rightarrow E_Z \rightarrow \ZZ(-p) \rightarrow 0$$
	and $[E_Z] \in \Ext_{MHS}(\ZZ(-p),H^{2p-1}(X))$. 
	
	This is a separated extension of mixed Hodge structures. From Carlson's theorem applied to the torsion free quotient  $\bar{H}^{2p-1}(X,\ZZ)=H^{2p-1}(X,\ZZ)/\{tors\}$ we can see that 
	$$\Ext_{MHS}(\ZZ(-p),\bar{H}^{2p-1}(X,\ZZ))$$
	is the same as the {\bf $p^{th}$ intermediate Jacobian.} 
	
	The {\bf  Abel-Jacobi map for codimensional-$p$ null-homologous cycles} is given by the association 
	$$Z \rightarrow [\bar{E}_Z] \in \Ext_{MHS}(\ZZ(-p),\bar{H}^{2p-1}(X,\ZZ)).$$
	where $\bar{E}_Z$ is the push-out of $E_Z$ under the map $H^{2p-1}(X,\ZZ) \rightarrow \bar{H}^{2p-1}(X,\ZZ)$.

	In particular, if $p=1$ this is the usual Abel-Jacobi map to $\Pic^0(X)$. Hence we have 
	
	\begin{lem} Let $X$ be a smooth projective curve. A divisor $Z$ of degree 0 is the divisor of a function if and only if the exact sequence
	$$0 \rightarrow \bar{H}^1(X,\ZZ) \rightarrow \bar{E}_Z \rightarrow \ZZ(-1) \rightarrow 0$$
	splits as an extension of MHSs. More generally it is torsion in the Jacobian if and only if the sequence of $\Q$-MHSs splits.  
	
	\end{lem}

	This argument can be generalized to the following 
	
	\begin{prop} Let $X$ be a smooth projective variety and $|D|=\cup D_i$ a union of codimensional $p$ subvarieties. Let $Y=X-|D|$. Let 
		$$H^{2p}_{|D|}(X,\Q)^0= \Ker\{H^{2p}_{|D|}(X,\Q) \rightarrow H^{2p}(X,\Q)\}$$

		Then every null-homologous codimension $p$ cycle $Z$ supported on $D$ is torsion in the $p^{th}$ intermediate Jacobian if and only if the exact sequence 
		$$0 \longrightarrow H^{2p-1}(X,\Q) \longrightarrow  H^{2p-1}(Y,\Q) \stackrel{\partial}{\longrightarrow} H^{2p}_{|D|}(X,\Q)^0 \longrightarrow 0$$
	
		splits.  
		\label{splitting}
		
	\end{prop}
	\begin{proof} If $Z$ is a null-homologous cycle supported on $|D|$ then it determines a morphism from $\Q(-p)$ to  $H^{2p}_{|D|}(X,\Q)^0$. If the exact sequence is split then the same splitting implies that the class of $E_Z$ is split. 
		
	Conversely, let $\{Z_i| i \in I\}$ be null-homologous cycles which form a basis for $H^{2p}_{|D|}(X,\Q)^0$.  If $E_{Z_i}$ is trivial for all $i$ then the sequence 
			$$0 \longrightarrow H^{2p-1}(X,\Q) \longrightarrow  H^{2p-1}(Y,\Q) \stackrel{\partial}{\longrightarrow} H^{2p}_{|D|}(X,\Q)^0 \longrightarrow 0$$
		is split. 
		\end{proof}
		
	For higher codimensional varieties it is not necessarily the case that a cycle which is trivial in the intermediate Jacobian is trivial in the Chow group. Hence from the splitting of the exact sequence one can only conclude that the cycle is $0$ in the intermediate Jacobian or is torsion if the $\Q$-mixed Hodge structure splits. A conjecture of Bloch \cite{blochinjectivity} and Beilinson asserts that if the variety is defined over a number field, such a cycle defined over the field is necessarily torsion in the Chow group as well.

		\section{Fermat varieties.}
		
		\subsection{Rohrlich's theorem and its generalizations.}
		
We first need some definitions. Let $K \subset \C$ be a  field containing the $d^{th}$ roots of unity with $d>1$. The {\bf Fermat hypersurface $F_d^n$ of dimension $n$ and degree $d$ over $K$} in $\CP^{n+1}$ is the hypersurface given by
		$$F^n_d:X_0^d+X_1^d+\dots+X_{n+1}^d=0$$
  For $I=\{i_1,\dots,i_p \} \subset \{0,\dots,n+1\}$ let 
  $$C_I=F^n_d \cap \bigcap_{i_k \in I} \{X_{i_k}=0\}$$
  $C_I$ is a subvariety isomorphic to the Fermat variety $F_d^{n-p}$ and lies in $CH^p(F_d^n)$. We call an irreducible component of  such a subvariety a {\bf coordinate hyperplane subvariety of codimension $p$}. If $n-p>0$ the coordinate hyperplane subvarieties are irreducible so only in the $n=p$ case there are components. However, later we will see that we can assume that $p=\frac{n+1}{2}$ and  this leaves us with only the case when $n=1$, namely the case of Rohrlich's theorem that is different. Let 
  $$S_d^p=|\bigcup_{|I|=p} C_I| $$
 be the support of the union of all the coordinate hyperplane subvarieties.

  From Proposition \ref{splitting}, 	in order to prove  Rohrlich's theorem we have to prove that the sequence 
  $$0 \rightarrow H^1(F_d,\Q) \rightarrow H^1(F_d-S_d,\Q) \rightarrow H^{2}_{S_d}(F_d,\Q)^0 \rightarrow 0$$ 
  is split as a sequence of mixed Hodge structures. Equivalently we need to find a complement to the image of $H^1(F_d,\Q)$ in $H^1(F_d-S_d,\Q)$ which is a pure Hodge structure of weight $2$.

  We have the following theorem 
  
  \begin{thm} Let $F^n_d$ be the Fermat hypersurface of dimension $n$ and degree $d$. Let $S_d^p=|\bigcup_{|I|=p}  C_I|$ be the support of the union of all the coordinate hyperplane subvarieties of codimension $p$.  Then the exact sequence 
	$$0 \longrightarrow H^{2p-1}(F^n_d,\Q) \longrightarrow H^{2p-1}(F^n_d-S_d^p,\Q) \stackrel{\partial}{\longrightarrow} H^{2p}_{S_d^p}(F_d^n,\Q)^0 \longrightarrow 0$$
	is a split exact sequence of Hodge structures. Here 
	$$H^{2p}_{S_d^p}(F_d^n,\Q)^0=\Ker\{H^{2p}_{S_d^p}(F^n_d,\Q) \rightarrow H^{2p}(F^n_d,\Q)\}$$
	is the kernel of the cycle class map and is generated by linear combinations of the coordinate hyperplane subvarieties  which are null-homologous.
	\label{nullhomologoustheorem} 
	\end{thm}
	
	We get the following corollary: 
 \begin{cor}
 	\label{nullhomologouscorollary}
	Any null homologous cycle supported on the set of coordinate hyperplane subvarieties  is torsion in the intermediate Jacobian. 
\end{cor}
In particular, when $n=1$ and $p=1$ this is Rohrlich's theorem. Note that for $n>1$ and $p<n$ the result is actually much stronger for simple reasons. In $CH^p(F_d^n)$, all the  $C_I \sim_{rat} H^p$ for all $I$, where $H$ is the hyperplane section. So if $\sum a_I [C_I]=0 \in H^{2p}(F_d^n,\Q)$  then $\sum a_I=0$, hence $(\sum a_I)H^p \sim_{rat} 0$ in $CH^p(F_d^n)$. So the only new result is an alternate proof of Rohrlich's theorem.

In order to prove the splitting we have to understand the cohomology of $F_d^n$. For that we follow Shioda \cite{shio}, who proved the Hodge conjecture for such varieties under some conditions. 

Let $\mu_d$ denote the group of $d^{th}$ roots of unity.  The group $\mu_d^{n+2}$ acts on $F_d^n$ by 
$$(\zeta_0,\dots,\zeta_{n+1})\cdot [X_0:\dots:X_{n+1}]=[\zeta_0X_0:\dots:\zeta_{n+1}X_{n+1}]$$
The diagonal acts trivially. Let $G_d^n=\mu_d^{n+2}/(diag)\simeq \mu_d^{n+1}$. The dual group of $G_d^n$  can be identified with 
$$\hG^n_d=\{\chi=(a_0,\dots,a_{n+1})|a_i \in \ZZ/d\ZZ \text{ such that } \sum a_i=0 \mod d\}$$ 
$$\chi((\zeta_0,\dots,\zeta_{n+1}))=\prod \zeta_i^{a_i}.$$

The action of $G^n_d$ on the Fermat varieties induces an action on related cohomology groups. Since $G^n_d$ also preserves $S_d^p$ it also induces a group action on the cohomology group $H^{2p}_{S_d^p}(F_d^n,\ZZ)$.

\subsection{$G_d^n$-module structure of $H^n(F_d^n,\C)$.}

We want to understand the decomposition of the cohomology groups $H^i(F^n_d,\Q)$, $H^i(F^n_d-S_d^p,\Q)$ and $H^i_{S_d^p}(F_d^n,\Q)$ as $G_d^n$ representations.

From the Lefschetz hyperplane theorem we know that if $0 \leq i \leq 2n$, $i \neq n$, 
$$H^i(F_d^n,\Q) = \begin{cases} \{0\} & i\neq n \text{ odd}\\
								\Q(-i/2) & i\neq n \text{ even} \end{cases}$$
and in the even case it is generated by $H^{i/2}$. Hence the only `interesting' cohomology is when $i=n$. 

Recall that the primitive cohomology is the orthogonal complement of the powers of the hyperplane class. From above, all the primitive cohomology groups are trivial for $i\neq n$. When $n$ is odd, $H^n_{prim}(F_d^n,\Q)\simeq H^n(F_d^n,\Q)$ but when $n$ is even the primitive cohomology is of codimension $1$ and orthogonal to $H^{n/2}$:
$$H^n(F_d^n,\Q)=H^n_{amb}(F_d^n,\Q) \oplus H^n_{prim}(F_d^n,\Q)$$
where the {\em ambient cohomology} is $H^n_{amb}(F_d^n,\Q)=\Q H^{n/2}$, where $H$ is the hyperplane section. This is  the part coming from the cohomology of the ambient $\CP^{n+1}$. Note that $G_d^n$ acts as the identity on $H^n_{amb}(F_d^n,\Q)$.

For a character $\chi=\chi(a_0,a_1,\dots,a_{n+1})$, let $V(\chi)$ denote the  corresponding  $1$-dimensional representation.  Shioda has the following theorem 
 
\begin{thm}[Shioda] Let $\A=\A_d^n \subset \hG_d^n$ be the subset 
	$$\A_d^n=\{\chi=\chi(a_0,\dots,a_{n+1})|\, a_i \neq 0 \text{ for all } i\}$$
	we have 	
 $$H^n_{prim}(F_d^n,\C)\simeq \bigoplus_{\chi \in \A^n_d} V(\chi)$$ 		
\label{shiodatheorem}
	\end{thm}
	Null-homologous cycles of codimension $p$  determine an extension in 
	$$\Ext^1_{MHS}(\Q(-p), H^{2p-1}(F_d^n,\Q)).$$
	Since $2p-1$ is odd, if $2p-1 \neq n$ then the cohomology group is trivial. So we can restrict ourselves to the case when $n=2p-1$. In  this case $H^n_{prim}=H^n$. 
	
This describes the action on the weight $(2p-1)$ part of the Hodge structure of $H^{2p-1}(F_d^n-S_d^p,\Q)$. In the next section we describe the action on the  weight $2p$ part. 

\subsection{$G_d^n$-module structure of $H^{2p}_{S_d^p}(F_d^n,\Q)$.}

To understand the $Gr_{2p}^{W_{\bullet}}$ part of the Hodge structure we have to understand the $G_d^n$ action on $H^{2p}_{S_d^p}(F_d^n,\Q)$. This is the $\Q$-vector space freely generated by the classes of irreducible components of $C_I$, where $I=\{i_0,\dots,i_{p-1}\}$ is a subset of $\{0,\dots,n+1\}$ of cardinality $|I|=p$.
$$H^{2p}_{S_d^p}(F_d^n,\Q) \simeq \bigoplus_{Z} \Q(-p)[Z]$$
where $Z$ runs through the irreducible components of $C_I$. Recall that we can assume $n=2p-1$ as $H^{2p-1}(F_d^n,\Q)=\{0\}$ if $n \neq 2p-1$. 

There are two cases:

\noindent Case 1: $n=1$. Here the coordinate hyperplane subvarieties  are the cuspidal points. Let $P$ be a cuspidal point in $C_I$ for some $I$. Here $C_I=C_{\{i\}}$ for  $i\in \{0,1,2\}$ so we will just use $i$. 

For any $P$ in $C_i$, the subgroup  
$$G_P=1\times\dots \times \mu_d \times \dots 1$$
with $\mu_d$ in the $i^{th}$ coordinate fixes $P$. The orbit of $P$ under the group $G_d^1$ is $C_i$. Note that $G_P$ is the same for all $P$ in $C_i$,  so we denote it by $G_i$. 

We would like to understand the representation $V_i$ corresponding to this orbit. Since $G_P=G_i$ acts trivially on $P$, the representation is 
$$V_i=\Ind_{G_i}^{G_d^1} ({\mathds 1})$$
From  Frobenius reciprocity the representation decomposes into those characters which restrict to the trivial character of $G_i$ 
$$V_i=\oplus_{\chi|_{G_i}={\mathds 1}} V(\chi)$$
These are those characters where $a_i=0$.

Let $B_i$ be the set of characters such that $a_i=0$. 
$$H^{2}_{S_d^1}(F_d^1,\C)=\bigoplus_{i=0}^{2}\bigoplus_{\chi \in B_i} V(\chi)$$
Note that $B_i \cap \A_d^1=\emptyset$ for all $i$. 

\vspace{\baselineskip}

\noindent Case 2: $n >1$. In this case the $C_I$ are irreducible. $G_d^n$ preserves $C_I$ for any $C_I$ so the action on the class of $C_I$ is as the  identity. Hence $G_d^n$ acts by multiplication by $1$ on $H^{2p}_{S_d^p}(F_d^n,\Q)$.

\subsection{Splitting of the Hodge Structure.}

We can now complete the proof of Theorem \ref{nullhomologoustheorem}.

\begin{proof} Let 
	
		$$0 \longrightarrow H^{2p-1}(F^n_d,\Q) \longrightarrow H^{2p-1}(F^n_d-S_d^p,\Q) \stackrel{\partial}{\longrightarrow} H^{2p}_{S_d^p}(F_d^n,\Q)^0 \longrightarrow 0$$
	be the exact sequence. 
	
	\vspace{\baselineskip}
	
\noindent{\bf Case I:} $n > 1, p<n$ -- the group $CH^p_{hom}(F_d^n)$.  Consider the element 
	$$e=\frac{1}{|G_d^n|} \sum_{g \in G_d^n} g$$
in $\Q[G_d^n]$  acting on $H^{2p-1}(F^n_d-S_d^p,\Q)$. Since 
$$H^{2p-1}(F_d^n,\C)\simeq \bigoplus_{\chi \in A^n_d} V(\chi)$$
for  $v \in V(\chi)$, 
 $$e(v)= \frac{1}{|G_d^n|}    \sum_{g \in G_d^n}  g\cdot v = \frac{1}{|G_d^n|}  \sum_{g \in G_d^n} \chi(g)v=0$$
 since $\chi$ is a non-trivial character. Hence $H^n(F_d^n,\C) \subset \Ker(e)$. Further, one has $e^2=e$ - that is, $e$ is a projector.  
 
$e$ also acts on $H^{2p}_{S_d^p}(F_d^n,\Q)$. Since the action of $g \in G_d^n$ preserves $C_I$, $e$ acts as the identity  on $H^{2p}_{S_d^p}(F_d^n,\Q)$. One also has 
$$\partial \circ e=e \circ \partial$$
Hence $\Ker(e)=\Ker(\partial)$ as we have already seen $\Ker(\partial) \subseteq \Ker(e)$ and if $v \in \Ker(e)$ then 
 $$\partial(v)=e(\partial (v))=\partial (e(v))=0$$
 so $v \in \Ker(\partial)$. 
 
Define a section
$$ H^{2p}_{S_d^p}(F_d^n,\Q)^0 \stackrel{s}{\longrightarrow} H^{2p-1}(F^n_d-S_d^p,\Q)$$
$$ s(v_0)=e(v) \text{ where } v \in \partial^{-1}(v_0)$$
This is well defined as if $v, v'\in \partial^{-1}(v_0)$, $v-v' \in H^{2p-1}(F_d^n,\Q)$ so $e(v-v')=0$. $s$ is injective as 
$$\partial(s(v_0))=\partial(e(v))=e(\partial (v))=e(v_0)=v_0$$
hence  $s(v_0)=0 \Rightarrow v_0=0$. 

$e$ is also a morphism of rational Hodge structures as it is a  $\Q$-linear combination of   endomorphisms induced by automorphisms of $F_d^n$. 

The image of $e$ is a pure Hodge substructure of weight $2p$ of $H^{2p-1}(F^n_d-S_d^p,\Q)$ and is isomorphic via $s$  to $H^{2p}_{S_d^p}(F_d^n,\Q)^0$.  Hence it provides a splitting of the exact sequence. 

\vspace{\baselineskip}
 
 \noindent{\bf Case 2:} $n=1$ -- the group $CH^1_{hom}(F_d^1)$. The argument is similar. Let 
 $$ e= \sum_{\chi \in B_0 \cup B_1 \cup B_2} e_{\chi}$$
 where $e_{\chi}$ is the projector corresponding to the character $\chi$. 
 While $e_{\chi}$ need not be defined over $\Q$  the sum is Galois invariant and lies in $\Q[G_d^1]$  and defines a morphism of rational mixed Hodge structures.  Since $B_i \cap \A_d^1=\emptyset$, $e$ projects on to the weight two piece of the Hodge structure on $H^1(F_d^1-S_d^1,\C)$ and hence provides a splitting along the same lines as above. 
 
 \vspace{\baselineskip}
 
 \noindent{\bf Case III:} $n>1,p=n$ -- the group $CH^n_{hom}(F_d^n)$. In this case the $C_I$ are no longer irreducible as they consist of $d$ points. Hence we have to apply Case II above to the $C_J$, with $|J|=n-1$ adding some auxiliary points if necessary and then combine them to get the result - as a degree $0$ cycle with support in $\bigcup_{|I|=n} C_I$ need not be of degree $0$ on the individual $C_J, |J|=n-1$ but we can add and subtract points in $C_J \cap C_{J'}$ to ensure that this is the case.  
 
 \end{proof}
 
 \begin{rem} The case when $n=2p-1$ is the only case where one expects potentially interesting cycles. For instance $CH^1_{hom}(F_d^1), CH^2_{hom}(F_d^3), CH^3_{hom}(F_d^5)$ and in fact there do exist interesting non-torsion cycles in some of these groups as we will see in Section \ref{nontorsion}.

\end{rem}

\section{Motivic cycles on Fermat varieties.}

A motivic cycle, or equivalently a higher Chow cycle, in the group $CH^p(X,q)=H^{2p-q}_{\M}(X,\ZZ(p))$ is represented by  a codimension $p$ cycle on $\Delta^q X=X \times \Delta^q$, where $\Delta^q$ is the $q$-simplex 
$$\Delta^q=\{(t_0,\dots,t_q) \in \AB^{q+1}| \sum t_i=1\}$$ 
subject to the condition that it intersects the faces properly and is closed with respect to the alternating face restriction map \cite{blocACKT,scholl}. 

A special case is when $q=1$. A cycle in the group $CH^p(X,1)=H^{2p-1}_{\M}(X,\ZZ(p))$ is represented by a sum 
$$\sum (Z_i,f_i)$$
where $Z_i$ are subvarieties of codimension $p-1$ and $f_i$ are functions on $Z_i$ subject to the condition that $\sum \div(f_i)=0$.

\subsection{Ambient decomposability.}
\label{ambientdecomp}

Let $\pi:F_d^n \rightarrow \Spec(K)$ be a Fermat hypersurface over a   field $K \subset \C$ containing the $d^{th}$ roots of unity, with $d>1$.  Let $p \geq q \geq 1$ and let $m=2p-q-1$. We  consider the rational higher Chow groups, $CH^p(F_d^n,q)_{\Q}$.

Let $H$ be the hyperplane section in $\CP^{n+1}$. A  cycle is said to be {\bf ambient decomposable} if it lies in the subspace 
$$CH^p_{amb}(F_d^n,q)_{\Q}=\Im \left( \bigoplus_{0 \leq r  \leq n}   CH^{p-r}(\Spec(K),q)_{\Q} \stackrel{\oplus H^r  \pi^*(-)}{\longrightarrow} CH^p(F_d^n,q)_{\Q}\right)$$
where $H$ is the hyperplane section and $\pi$ is the structure map. In other words these are the kind of cycles determined by the embedding and the base field - and not the variety itself. We define a cycle $y$ to be {\bf regulator ambient decomposable} if
$$\reg(y)=\reg(y_{amb})$$
for some ambient decomposable cycle $y_{amb}$. 

If $y=\sum H^r \pi^*(\beta_r)$ is an ambient cycle then 
$$\beta_r=\frac{1}{d} \pi_*(H^{n-r}.y)$$
so $y=0 \Leftrightarrow \beta_r=0$ for all $r$ and the map from the direct sum is injective.

\subsubsection*{The ambient correspondence.}

More generally, we can define an algebraic correspondence -- 
$$P_{amb}=\frac{1}{d} \sum_{j=0}^n H^{n-j} \times H^j \in CH^n(F_d^n \times F_d^n)_{\Q}$$
$$P_{amb}(y)=\sum_{j=0}^n H^j \pi^*(\beta_j) \hspace{1in} \beta_j=\frac{1}{d} \pi_*(H^{n-j}y)$$
One has $P_{amb}^2=P_{amb}$ - so it is a projector - and the image is precisely the ambient cycles.

Let $G_i$ as before be the subgroup of $G=G_d^n$  consisting of those elements $g$ which are $1$ in all coordinates except for the $i^{th}$ one - so $G_i \simeq \mu_d$. 

One has $F_d^n/G_i \simeq \CP^{n}$ and this is the same as the map
$$F_d^n \stackrel{\rho_i}{\longrightarrow} \CP^{n}$$
$$[x_0,\dots,x_{n+1}] \stackrel{\rho_i}{\longrightarrow} [x_0,\dots,\hat{x}_i,\dots, x_{n+1}]$$
obtained by forgetting the $i^{th}$ coordinate. This is well defined on $F_d^n$ as the only possible point of indeterminacy is the point where  $x_j=0, j \neq i$ but this does not lie on $F_d^n$. 

The map $\rho_i$ is a finite flat map of degree $d$ and if $h$ is the hyperplane section in $\CP^n$ then $\rho_i^*(h)=H$. 

The algebraic correspondence defining $\rho_i^* \rho_{i,*}$ is given by 
$$F_d^n \times_{\CP^n} F_d^n \subset F_d^n \times F_d^n.$$
Over the locus where $\rho_i$ is \'etale, this is given by the graphs of elements of $G_i$. Taking the closure gives 
$$[F_d^n \times_{\CP^n} F_d^n]=\sum_{g \in G_i} [\Gamma_g]$$
The ramification locus does not introduce any new components of dimension $n$. Hence one has 
$$\rho_i^* \rho_{i,*}=\sum_{g \in G_i} g^*$$
$$T_i=\frac{1}{d} \rho_i^* \rho_{i,*}$$

On the other hand, in $CH^n(\CP^n \times \CP^n)_{\Q}$, 
$$[\Delta_{\CP^n}]=\sum_{j=0}^n h^j \times h^{n-j}$$
The flat pullback by $\rho_i \times \rho_i$ gives 
$$[F_d^n \times_{\CP^n} F_d^n]=(\rho_i \times \rho_i)^*(\Delta_{\CP^n})=\sum_{j=0}^n H^j \times H^{n-j} $$
Dividing by $d$  shows that $T_i=P_{amb}$ as algebraic correspondences in $CH^n(F_d^n \times F_d^n)_{\Q}$. Note that the right hand side is independent of $i$.

\subsubsection*{The primitive correspondence.} 

Recall we have an idempotent $e_{prim}$ in $\Q[G_d^n]$ which projects on to the primitive part. 
$$e_{prim}=\sum_{\chi \in \A} e_{\chi}$$
where $e_{\chi}$ is the projector onto the $\chi$-eigenspace 
$$g(e_{\chi}(v))=\chi(g)e_{\chi}(v)$$

Let $t_i=\frac{1}{|G_i|}\sum_{g \in G_i} g$. We claim 
$$e_{prim}=\prod_{i=0}^{n+1} (1-t_i)$$
To see this, if $\chi$ is a character then on the $\chi$-eigenspace $V(\chi)$, $t_i$ acts by $1$ if $\chi|_{G_i}=1$ and $0$ otherwise. So $\prod_{i=0}^{n+1}(1-t_i)=0$ on $V(\chi)$ if $\chi|_{G_i}=1$ for some $i$ and $1$ otherwise. Hence $\prod_{i=0}^{n+1}(1-t_i)=\sum_{\chi \in \A} e_{\chi}=e_{prim}$. 

An element of the group ring $\Q[G_d^n]$ determines an algebraic correspondence 
$$ \Q[G_d^n] \stackrel{\Phi}{\longrightarrow} CH^n(F_d^n \times F_d^n)_{\Q}$$ 
$$\sum a_g g  \rightarrow \sum a_g [^t \Gamma_g]$$
Multiplication in the group ring is taken to composition of correspondences. If $E_{prim}$ is the associated correspondence to $e_{prim}$ then 
$$E_{prim}=\Phi(e_{prim})=\circ_{i=0}^{n+1} \left([\Delta_{F_d^n}] - \frac{1}{d} \sum_{g \in G_i} [^t \Gamma_g]\right)= \circ_{i=0}^{n+1} \left([\Delta_{F_d^n}] - P_{amb}\right)$$
since $P_{amb}^2=P_{amb}$ this simplifies to 
$$E_{prim}=[\Delta_{F_d^n}]-P_{amb}$$
Therefore if $y$ is in $CH^p(F_d^n,q)_{\Q}$ then 
$$y=E_{prim}(y)+P_{amb}(y)$$

\subsection{Cycles from coordinate hyperplane subvarieties.}

We can use Corollary \ref{nullhomologouscorollary} to construct cycles in these groups. As an example, we construct a cycle in $CH^p(F_d^n,1)$.

Let  $Z$ be a coordinate hyperplane subvariety of the type $C_I$ for some $I \subset \{0,1,\dots,n+1\}$ as before. Let $\{J\}$ be the sets such that $I \subset J$ and  $|J|=|I|+1$. Any null homologous divisor on $C_I$ supported  on the $C_J$ is torsion in the Jacobian as if $\dim(C_I)>1$ then $H^1(C_I,\Q)=\{0\}$ and if $\dim(C_I)=1$ we can apply Rohrlich's theorem. Hence it is not very difficult to find functions. 

Using these functions we can construct higher Chow cycles. For instance, in the group $CH^p(X,1)$ with $2 \leq p \leq n$,  a cycle is represented by a sum $\sum (Z_i,f_i)$ with $\sum \div(f_i)=0$. We proceed as follows:

Let $I=\{i_1,\dots,i_{p-2}\}$ be a subset of $\{0,\dots,n+1\}$. Let $a,b,c$ be three elements of the complement of $I$ in $\{0,\dots,n+1\}$ and consider the coordinate hyperplane subvarieties of the type  $C_{a}=C_{I \cup \{a\}}$  and $C_{ab}=C_{I \cup \{a,b\}}$. The coordinate hyperplane subvarieties $C_a$ and $C_b$ meet at $C_{ab}=C_{ba}$.

On $C_a$ we have a function $f_{bc}$ with
$$\div(f_{bc})=C_{ab}-C_{ac}.$$
We have similar functions on $C_b$ and $C_c$. The functions can be given explicitly - for instance  $f_{bc}=\frac{X_b}{X_c}|_{C_a}$.

We then have 
$$\xi=(C_a,f_{bc})+(C_b,f_{ca})+(C_c,f_{ab})$$ 
is an element of $CH^{p}(F_d^n,1)$ as 
$$\div_{C_a}(f_{bc})+\div_{C_b}(f_{ca})+\div_{C_c}(f_{ab})$$
$$ =C_{ab}-C_{ac}+C_{bc}-C_{ba}+C_{ca}-C_{cb}=0$$
More generally one can construct cycles in $CH^p(F_d^n,q)$ along the same lines. Ross \cite{ross} constructed cycles in the group $CH^2(F_d^1,2)$ using pairs of functions supported on the Fermat points. 

In fact $\xi$ is ambient decomposable! On $C_I$ let 
$$u=\frac{X_c}{X_a}, \qquad v=\frac{X_c}{X_b}$$
Recall that the tame symbol of $\{u,v\}$ at an irreducible codimensional $1$ subvariety $C$ of $C_I$  is 
$$\tau (\{u,v\})|_C=(-1)^{\ord_C(u)\ord_C(v)}\frac{u^{\ord_C(v)}}{v^{\ord_C(u)}}.$$
The tame  symbol of $\{u,v\}$ is supported on $C_a,C_b$ and $C_c$. Computing shows 
$$\tau(\{u,v\})= (C_a,v|_{C_a})+(C_b,u^{-1}|_{C_b})+(C_c,-u/v|_{C_c})$$
$$=(C_a, f_{bc}^{-1})+ (C_b,f_{ca}^{-1})+(C_c,-f_{ab}^{-1})=-\xi+(C_c,-1)$$
Since the image of the tame symbol is $0$ in the Chow group the cycle $\xi\sim_{rat} (C_c,-1)$ and is ambient decomposable! In fact, in this case $2(C_c,(-1)) \sim_{rat} 0$ so $2\xi \sim_{rat} 0$.  

In the subsequent sections we show  that cycles determined by  coordinate hyperplane subvarieties are  ambient decomposable -- so the fact that the cycle above is ambient decomposable is not surprising.  
In general we cannot say the cycles are  torsion though under some conditions this is forced as well.

\subsection{Ambient decomposability of coordinate hyperplane cycles.}

We now show that certain cycles constructed from coordinate hyperplane subvarieties of $F_d^n$ can be seen to be ambient decomposable or at least regulator  ambient decomposable in some cases. This is the higher Chow version of Rohrlich's theorem.
 
We consider cycles in the group $CH^p(F_d^n,q)_{\Q}$ with $p-q \geq 0$, $n>1$. We first observe that if $n \neq 2p-q-1$ then 
$$H^{2p-q-1}(F_d^n,\Q)_{prim}=0\text{ and } H^{2p-q-1}(F_d^n,\Q)=H^{2p-q-1}_{amb}(F_d^n,\Q)$$
The regulator map is 
$$CH^p(F_d^n,q)_{\Q} \stackrel{\reg}{\longrightarrow} \Ext_{\Q-MHS}(\Q(-p),H^{2p-q-1}(F_d^n,\Q))=\Ext_{\Q-MHS}(\Q(-p),H_{amb}^{2p-q-1}(F_d^n,\Q))$$
By functoriality of the regulator, for any cycle $y$, 
$$\reg(y)=P_{amb}(\reg(y))=\reg(P_{amb}(y))$$ 
and $y_{amb}=P_{amb}(y) \in CH^p_{amb}(F_d^n,q)_{\Q}$ so $y$ is regulator ambient decomposable. Standard conjectures on injectivity of the regulator map if the base field $K$ is a number field and the cycles are defined over $K$  would imply that the cycles are actually ambient decomposable. 

In what follows we show that if we make a restriction on the support of the cycles we can show that the cycles are actually ambient decomposable. This holds regardless of the condition that $2p-q-1 \neq n$. 

 Consider codimensional $p$ cycles $y \subset F_d^n \times \Delta^q$ which are closed and meet the faces properly - thus determine an element of $CH^p(F_d^n,q)_{\Q}$. Further if  $\Gamma_0$ is the support, every irreducible component of $\Gamma_0$ projects on to $C_I$ for some $I$ with $|I|\geq p-q>0$.   

\begin{thm} Let $y$ be a cycle in $CH^p(F_d^n,q)_{\Q}$ supported on coordinate hyperplane subvarieties as described above. Suppose $p-q>0$. Then 
$$E_{prim}(y)=0 \text{ and } P_{amb}(y)=y \qquad \text{ so } y\in CH^p_{amb}(F_d^n,q)_{\Q}$$
\label{caseI}
\end{thm}

\begin{proof}

Let $\Gamma$ be an irreducible component of a higher Chow cycle $y$ . If $G_{\Gamma}$ is the stabilizer of $\Gamma$ then  the representation of $G_d^n$ on the orbit of $\Gamma$ is
$$\Ind_{G_{\Gamma}}^{G_d^n} {\mathds 1}$$
From our hypothesis, $\Gamma$ projects onto $C_I$ for some $I$. For $i \in I$, $G_i$ fixes $C_I$  and acts trivially on the simplex coordinates and hence fixes $\Gamma$. Therefore $G_i \subset G_{\Gamma}$. 

Since $G_i \subset G_{\Gamma}$ for all $i$ in $I$ the characters appearing in $\Ind_{G_{\Gamma}}^{G_d^n} {\mathds 1}$ are all trivial on $G_i$ - which means $a_i=0$. Hence they are distinct from the characters appearing in $\A$. Hence $e_{prim}(\Gamma)=0$. Since the stabilizer of every component contains some $G_i$ this is true for all components and 
$$e_{prim}(y)=0=E_{prim}(y)$$ 
Since $P_{amb}=1-E_{prim}$ this implies $P_{amb}(y)=y$. 
From the expression of $P_{amb}$ one has 
$$y=P_{amb}(y)=\sum_{j=0}^n H^j \pi^*(\beta_j)$$
with $\beta_j \in CH^{p-j}(\Spec(K),q)_{\Q}$. If $p-j>q$ then $CH^{p-j}(\Spec(K),q)_{\Q}=0$ so only the case when $j \geq p-q$ can contribute and so $y=H^{p-q}.z$ where $z \in CH^q(F_d^n,q)_{\Q}$.

\end{proof}

\noindent {\bf Case II: $p-q=0$ and $n=2p-q-1=p-1>1$}

In this case the higher Chow group is the group $CH^p(F_d^{p-1},p)$ and the Fermat variety is at least $2$ dimensional. 

Let $C_i$ be the coordinate hyperplane subvarieties corresponding to $X_i=0$. Let $C=\bigcup C_i$ and $U=F_d^{p-1}-C$. Let $i:C \rightarrow F_d^{p-1}$ and $j:U \rightarrow F_d^{p-1}$.

We consider cycles  $y \in CH^p(F_d^{p-1},p)_{\Q}$ such that 
$$y|_U=j^*(y)=\sum c_{s}\{f_{s_1},\dots,f_{s_p}\} \hspace{2cm} c_s \in \Q$$
where $f_{s_i}$ are functions in $\OO(U)^*$ with divisor supported on $C$. We say such cycles satisfy the {\em symbol hypothesis}. 

We then have the following theorem 

\begin{thm}  Let $y$ be a cycle in the group $CH^{p}(F_d^{p-1},p)_{\Q}$ satisfying the symbol hypothesis with $p>2$. Then 
	$$E_{prim}(y)=0 \qquad P_{amb}(y)=y \in CH^p_{amb}(F_d^{p-1},p)_{\Q}$$
	
	\end{thm}

\begin{proof}

Since $n=p-1>1$, every $C_i$ is irreducible and is preserved by the $G_d^n$ action. Therefore 
$$\div(g^*(f))=\div(f)$$
for every such $f$. Since $F_d^n$ is projective this implies that $g^*(f)/f=c_g$ for some constant $c_g \in K^*$. Since $g^d=1$ we further have $c_g \in \mu_d$. Hence for $\alpha=\{f_1,\dots,f_p\}$ in $CH^p(U,p)_{\Q}$ and $g$ in $G_d^n$, 
$$g^*(\alpha)-\alpha=\{c_1 f_1,\dots,c_p f_p\}-\{f_1,\dots,f_p\}$$
This is $d$-torsion - so $g^* \alpha=\alpha$ in $CH^p(U,p)_{\Q}$. 
By multilinearity, this  implies 
$$g^*(y|_U)=y|_U$$
for any $y \in CH^p(F_d^{p-1},p)_{\Q}$ satisfying the symbol hypothesis. That is, the action of $G_d^{p-1}$ on the restrictions to $U$ of such cycles is trivial.  

Since the trivial character is not in $\A$, this means that for such $y$
$$E_{prim}(y)|_U=e_{prim}(y|_U)=e_{prim}(j^*(y))=0.$$

Localization gives the exact sequence 
$$CH^{p-1}(C,p)_{\Q} \stackrel{i_*}{\longrightarrow} CH^p(F_d^{p-1},p)_{\Q} \stackrel{j^*}{\longrightarrow} CH^p(U,p)_{\Q}$$
The kernel of $j^*$ is $i_*(CH^{p-1}(C,p)_{\Q})$. These are represented by cycles with support in $C$. 

If $W$ is an irreducible component of some cycle in $Z^{p-1}(C,p)_{\Q}$  then $W \subset C_i \times \Delta^p$ for some  $i$. Hence the group $G_i$ fixes $W$ as it fixes $C_i$ and acts trivially on the simplex. 

Hence $t_i(W)=W$ for $W \subset C_i \times \Delta^p$. Since   $e_{prim} \circ t_i=0$  one has 
$$e_{prim}(W)=0$$
This implies $e_{prim}(CH^{p-1}(C,p)_{\Q})=0$.

Hence if  $y$ is  a cycle satisfying the symbol hypothesis 
$$j^*(e_{prim}(y))=e_{prim}(j^*(y))=0$$
so $e_{prim}(y) \in \Ker(j^*)=\Im(i_*)$. However 
$$e_{prim}(\Im(i_*))=0$$
so 
$$e_{prim}(e_{prim}(y))=e^2_{prim}(y)=e_{prim}(y)=E_{prim}(y)=0$$
for any $y$ satisfying the symbol hypothesis. 

\end{proof}

When $p=2$, namely the case of curves, the argument fails as the $C_i$ are no longer irreducible. 

\subsubsection*{Fermat varieties over number fields} 

If we assume $K$ is a number field and $q>0$  then there is a little more that we can say. One has (\cite{degl}, Theorem 5.1.1),
$$CH^a(\Spec(K),q)_{\Q}=0 \qquad \text{unless }2a-q=1$$
From this one has: 

\begin{itemize}
\item If $q$ is even, the ambient subgroup $CH^p_{amb}(F_d^n,q)_{\Q}=0$. 

\item If $q$ is odd and $r=p-(q+1)/2$ 

\begin{itemize}
	
	\item If  $r \in [0,n]$, then the ambient subgroup is precisely 
$$CH^p_{amb}(F_d^n,q)_{\Q}=H^r \pi^*(CH^{(q+1)/2}(\Spec(K),q)_{\Q})$$ 
and if $y$ is an ambient cycle and $\beta=\frac{1}{d} \pi_*(H^{n-r}.y)$ then 
$$y=H^r \pi^*(\beta).$$

\item if $r \notin [0,n]$ then $CH^p_{amb}(F_d^n,q)_{\Q}=0$.
\end{itemize}

\end{itemize}

	\begin{rem} The argument does not work when $n=1$, that is, in the case of Fermat curves, as in that case the divisor of a function is supported on points that are not necessarily fixed by the action of $G_d^1$ since the $0$-dimensional coordinate hyperplane subvarieties are not irreducible. 
		
  In fact there are examples of non-trivial cycles. Ross \cite{ross}  constructed cycles in the group $CH^2(F_d^1,2)$ which are non-trivial using the functions constructed by Rohrlich \cite{rohr}. 
		\end{rem}

\begin{rem} Again, like in the case of null-homologous cycles, the only really interesting higher Chow groups are $CH^p(F_d^{2p-q-1},q)$,  for example $CH^2(F_d^2,1),\, CH^2(F_d^1,2),\, CH^3(F_d^4,1), \dots$. These are the cases where the Deligne cohomology is determined by the middle cohomology group. 
	
What makes a cycle `interesting', at least from the point of view of special value conjectures, is having a non-torsion primitive part. The results of this section show that cycles coming from coordinate hyperplanes or satisfying the symbol hypothesis have torsion primitive part. They need not be torsion but at some level are `uninteresting'. 	
	
In Section \ref{nontorsion} we will see that there are possibly non-torsion primitive cycles in these groups coming from the product structure. 
	\end{rem}
	
	\begin{rem} Ambient decomposability also implies splitting of certain Hodge structures. Analogous to the description of Abel-Jacobi map the regulator maps can be described in terms of extensions of mixed Hodge structures. The results above imply splitting of certain exact sequences involving the primitive cohomology analogous to the splitting of the exact sequence in Theorem \ref{nullhomologoustheorem}. 
		
		\end{rem}

\section{Non-torsion cycles on  Fermat varieties.}
\label{nontorsion}

The above theorems show that certain `natural' constructions do not produce interesting cycles on Fermat varieties. This does not mean that the Chow groups of Fermat varieties are uninteresting.  

There is an inductive structure on Fermat Varieties which has interesting implications for  cycles on Fermat varieties.  In order to define it we need to assume that the field of definition $K$ is large enough to contain the $d^{th}$ roots of $-1$.

There is a rational map 
$$F_d^r \times F_d^s \dashrightarrow F^{r+s}_d$$
defined by 
$$[x_0:\dots:x_{r+1}],[y_0:\dots:y_{s+1}] \dashrightarrow [x_0 y_{s+1}:\dots:x_r y_{s+1}:\epsilon x_{r+1}y_0:\dots:\epsilon x_{r+1}y_s]$$ 
where $\epsilon^{d}=-1$. As a result we have a map 
$$F_d^1 \times F_d^1 \times F_d^1 \stackrel{\phi}{\dashrightarrow} F_d^3$$
and this can be lifted to a morphism $\tilde{\phi}:Y \rightarrow F_d^3$ where $Y$ is a blow up resolving the indeterminacy.

\[
\begin{tikzcd}
	Y \arrow[d, "b"'] \arrow[dr, "\widetilde{\phi}"] & \\
	(F_d^1)^3 \arrow[r, dashed, "\phi"'] & F_d^3
\end{tikzcd}
\]

In the group $CH^2_{hom}((F_d^1)^3)_{\Q}$ there is a well known cycle called the modified diagonal cycle of Gross and Schoen \cite{grsc}. If $e$ is a point on $F_d^1$, 
$$\Delta_e=\Delta-\Delta_{12}-\Delta_{13}-\Delta_{23}+\Delta_1+\Delta_2+\Delta_3$$
where $\Delta_{ij}$ and $\Delta_i$ are diagonals with $e$ in the complementary positions - for instance $\Delta_{12}=(x,x,e)$ where $x\in F_d^1$. 

This is called the Gross-Schoen modified diagonal cycle. It is well known that this cycle is in general non-trivial in the Chow group of null-homologous cycles on $(F_d^1)^3$ \cite{otsu}. 

We can consider the image 
$$\tilde{\phi}_*(b^*(\Delta_e))$$
to get a null-homologous cycle in $CH^2_{hom}(F_d^3)$. It is not unreasonable to expect that this cycle is non-torsion. More generally there are modified diagonal cycles in $CH^p_{hom}((F_d^1)^{2p-1})$ studied by Gross-Schoen \cite{grsc} and more recently by Eskandari \cite{eska}. It is likely that these cycles will also map to non-torsion null homologous cycles on $F_d^{2p-1}.$ 

Similarly, a construction of Bloch's \cite{blocirvine} provides cycles in $CH^2(F_d^1 \times F_d^1,1)$. The push-forward of these cycles  to the Jacobian provides cycles in $CH^g(J(F_d^1),1)$, where $g$ is the genus of $F_d^1$, which were studied by Sarkar \cite{sark}. He provided numerical evidence for them to be non-zero for certain $d$. One might expect that the push-forward of the same cycles to $F_d^2$ using the inductive structure could provide non-trivial cycles. 

As we mentioned earlier, Ross \cite{ross} constructed non-torsion cycles in the groups $CH^2(F_d^1,2)$ using functions supported on the cusps. However he also showed that these cycles cannot provide all the cycles expected by the Beilinson conjectures when $d$ is a large enough prime.

There are other results of Kimura \cite{kimu} and more recently Nemoto \cite{nemo} on non-torsion of null-homologous cycles in related Fermat varieties.

\bibliographystyle{alpha}
\bibliography{Fermatvarieties.bib}

\end{document}